\documentclass[11pt,leqno]{article}

\usepackage{amsmath, amsfonts, theorem,color,
}
\usepackage{latexsym}
\usepackage{graphicx}
\makeatletter
\def\@maketitle{\newpage
    \null
    \vskip .8truein
    \begin{center}%
     {\bf \@title \par}%
     \vskip 1.5em
     {\small
      \lineskip .5em
      \begin{tabular}[t]{c}\@author
      \end{tabular}\par}%
    \end{center}%
    \par
    \vskip .4truein}
\@addtoreset{equation}{section} \@addtoreset{theorem}{section}
\@addtoreset{lemma}{section} \@addtoreset{proposition}{section}
\@addtoreset{definition}{section} \@addtoreset{corollary}{section}
\@addtoreset{remark}{section}

\def\dfrac#1#2{\ds{\frac{#1}{#2}}}

\newcommand{\re}{{\mathbb R}}

\let\ds=\displaystyle

\def\R{{\mathbb R}}

\newtheorem{theorem}{Theorem}[section]
\newtheorem{lemma}{Lemma}[section]
\newtheorem{proposition}{Proposition}[section]
\newtheorem{definition}{Definition}[section]

\newtheorem{remark}{Remark}[section]
\newtheorem{example}{Example}[section]

  {\hfill$\Box$\bigskip\par}

\DeclareMathOperator{\diver}{div}
\def\proof{\list{}{\setlength{\leftmargin}{0pt}
                      \parskip=0pt\parsep=0pt\listparindent=2em
                      \itemindent=0pt}\item[]\futurelet\testchar\@maybe}

\def\@maybe{\ifx[\testchar \let\next\@Opt
          \else \let\next\@NoOpt \fi \next}
\def\@Opt[#1]{{\it Proof of #1.\ }}\def\@NoOpt{{\it Proof.\ }}
\begin{document}
\title{\Large \bf A note on non-coercive first order Mean Field Games with analytic data}
\author{{\large \sc Paola Mannucci\thanks{Dipartimento di Matematica ``Tullio Levi-Civita'', Universit\`a di Padova, mannucci@math.unipd.it}, Claudio Marchi \thanks{Dipartimento di Ingegneria dell'Informazione, Universit\`a di Padova, claudio.marchi@unipd.it},}\\
 {\large \sc Carlo Mariconda\thanks{Dipartimento di Matematica ``Tullio Levi-Civita'', Universit\`a di Padova, maricond@math.unipd.it},  Nicoletta Tchou\thanks{
IRMAR, Universit\'e de Rennes, CNRS, IRMAR - UMR 6625, F-35000 Rennes, France, nicoletta.tchou@univ-rennes1.fr}}
}

\maketitle
\begin{abstract}
We study first order evolutive Mean Field Games whose operators are non-coercive. This situation occurs, for instance, when some directions are ``forbidden'' to the generic player at some points. Under some very strong regularity assumptions, we establish existence of a weak solution of the system. Mainly, we shall describe the evolution of the population's distribution as the push-forward of the initial distribution through a flow, suitably defined in terms of the underlying optimal control problem. 
As a byproduct, we shall prove the uniqueness of the optimal trajectories associated to the optimal control underlying the Hamilton-Jacobi equation for every initial point and every starting time.
\end{abstract}


%
\section{Introduction}

In this paper we study the following Mean Field Game (briefly, MFG) system
\begin{equation}\label{eq:MFG1}
\left\{\begin{array}{lll}
(i)&\quad-\partial_t u+H(x, Du)=F(x,t, m)&\qquad \textrm{in }\re^d\times (0,T)\\
(ii)&\quad\partial_t m-\diver (m\, \partial_pH(x, Du))=0&\qquad \textrm{in }\re^d\times (0,T)\\
(iii)&\quad m(x,0)=m_0(x), u(x,T)=G(x, m(T))&\qquad \textrm{on }\re^d,
\end{array}\right.
\end{equation}
where, if $p=(p_1,\cdots, p_d)$ and $x=(x_1, \cdots, x_d)$, the functions $H(x,p)$ is
\begin{equation}\label{H}
H(x,p)=\frac{1}{2}|pB(x)|^2
\end{equation}
with $B(x)=B(x_1,...,x_d)$ equal to the $d\times d$ matrix 
\begin{equation}\label{hmatrix}
   \begin{pmatrix}
      \!\!&h_{11} & 0&0&0&\cdots &0\!\\
     \!\!&h_{21}(x_1)&h_{22}(x_1) &0&0&\cdots &0\!\\
      \!\!&h_{31}(x_1, x_2)&h_{32}(x_1, x_2)&h_{33}(x_1, x_2)&0&\cdots&0\!\\
      \!\!&\cdots &\cdots  &\cdots&\cdots&\cdots&\cdots\!\\
      \!\!&h_{d1}(x_1,\cdots,x_{d-1})&h_{d2}(x_1,\cdots,x_{d-1})&h_{d3}(x_1,\cdots,x_{d-1})&\cdots \!\!&\cdots&h_{dd}(x_1,\cdots,x_{d-1})\!
\end{pmatrix}
    \end{equation}
where $h_{ij}=h_{ij}(x_1, x_2,\cdots,  x_{i-1})$ and $h_{11}$ is a non-null constant.
From \eqref{H} and \eqref{hmatrix} we have $\partial_pH(x,p)= p\,B(x)\,B^T(x)$.

We shall assume that the functions ~$h_{ij}(x)$ are analytic, bounded, {\em possibly vanishing}
and that $F$ is a nonlocal strongly regularizing coupling (see assumptions below). 
In our paper \cite{MMMT2}, we considered the case $d=2$ with
 much weaker assumptions on the regularity of the coefficients. In our forthcoming paper
 \cite{MMMT3} we shall tackle the case of unbounded coefficients in a generic d-dimensional space.\\
 The main nolvelty with respect to our paper \cite{MMMT2} is the following: in this paper, taking advantage of the strong regularity of the coefficients, we are able to prove the uniqueness of the optimal trajectories (for the control problem underlying the Hamilton-Jacobi equation) for any initial point and any initial time.

These MFG systems arise when the dynamics of the generic player are deterministic and, in some points, may have ``forbidden'' directions (for instance, when the $h_{ij}$ vanish); actually, if the evolution of the whole population's distribution~$m$ is given, each agent wants to choose the control $\alpha=(\alpha_1,\cdots,\alpha_d)$ in the set
\begin{equation}\label{A}
{ \cal{A}}=\{\alpha:[t,T]\rightarrow \re^d, \alpha\in L^2([t,T]) \}
\end{equation}
in order to minimize the cost
\begin{equation}\label{Jgen}
\int_t^T\left[\frac12 |\alpha(\tau)|^2+F(x(\tau),\tau,m)\right]\,d\tau+G(x(T), m(T))
\end{equation}
where, in $[t,T]$, its dynamics $x(\cdot)$ are governed by
  \begin{equation}\label{eq:HJgen}
\left\{
\begin{array}{ll}
x'(s)=\alpha(s)B^T(x)& \\
x(t)=x.&
\end{array}\right.
\end{equation}
We study a problem with dynamics as in \eqref{eq:HJgen} because the structure of the matrix $B$ in 
\eqref{hmatrix} allows us to simplify the calculations in Section \ref{sect:HJ}.
Moreover this class of problems encompasses the Grushin and the Carnot type model (as the Heisenberg one, see Examples \ref{ex:Gru} and \ref{ex:heisen} below). 

Let us recall that the MFG theory studies Nash equilibria in games with a huge number of (``infinitely many'') rational and indistinguishable agents. This theory started with the pioneering papers by Lasry and Lions \cite{LL1,LL2,LL3} and by Huang, Malham\'e and Caines \cite{HMC}. A detailed description of the achievements obtained in these years goes beyond the scope of this paper; we just refer the reader to the monographs \cite{AC, C, BFY, GPV, GS}.
As far as we know, degenerate MFG systems have been poorly investigated up to now. Dragoni and Feleqi~\cite{DF} studied a second order (stationary) system where the principal part of the operator fulfills the H\"ormander condition; moreover, Cardaliaguet, Graber, Porretta and Tonon~\cite{CGPT} tackled degenerate second order systems with coercive (and convex as well) first order operators. Hence, these results cannot be directly applied to the non-coercive problem~\eqref{eq:MFG1}.
On the other hand, the existence of a solution to~\eqref{eq:MFG1} can be obtained by the vanishing viscosity approach as in \cite[Sect. 4.4]{C} (see also~\cite[Sect. 2.5]{LL3}).
Unfortunately, the vanishing viscosity method seems to give no interpretation for the solution to the system.

Therefore, we shall pursue a different approach in order to obtain more detailed information on the evolution of the population's density. In particular, following the arguments in~\cite[Sect. 4.3]{C}, we are able to describe this evolution as the push-forward of the distribution at the initial time through a flow which is suitably defined in terms of the optimal control problem underlying~\eqref{eq:MFG1}. As a matter of facts, the non-coercivity of~$H$ prevents from applying directly the arguments of \cite[Sect. 4.3]{C}. The well-posedness and several properties of this flow may be considered among the main novelties of this paper. In particular, we prove that the optimal trajectories of the control problem associated to the Hamilton-Jacobi equation~\eqref{eq:MFG1}-(i) are unique after the starting time (see Proposition~\ref{thmunicita} below); as far as we know this property has never been tackled before for degenerate dynamics as~\eqref{eq:HJgen}. The proof of this property is the only point where our very strong regularity assumptions play a role. The rest of the paper follows the argument of \cite{MMMT2}.

We now list our notations and the assumptions that will hold throughout the whole paper, we give the definition of (weak) solution to system~\eqref{eq:MFG1} and we state the main results for system~\eqref{eq:MFG1}: 
existence and representation formula.

\noindent\underline {Notations and Assumptions}. We denote by~$\mathcal P_1$ the space of Borel probability measures on~$\re^d$ with finite first order moment, endowed with the Kantorovich-Rubinstein distance~{${\bf d}_1$}. Throughout this paper, we shall require the following hypotheses:
\begin{itemize}
\item[(H1)]\label{H1} for every $m\in C([0,T],\mathcal P_1)$, the function~$F(\cdot,\cdot,m)$ is analytic;~$G$ is a real-valued function, continuous on~$\re^d\times\mathcal P_1$;
\item[(H2)]\label{H2} there exists~$C$ such that
$$\|F(\cdot,t,m_1)\|_{C^2}, \|G(\cdot,m_2)\|_{C^2}\leq C,\qquad \forall m_1\in C^0([0,T],\mathcal P_1),\, m_2\in \mathcal P_1;$$
\item[(H3)]\label{H3} the functions of the matrix B, $h_{ij}:\mathbb R^{i-1}\to\mathbb R$ are analytic, with $\|h_{ij}\|_{C^2}\leq C$;
\item[(H4)]\label{H4} the initial distribution~$m_0$ is absolutely continuous with a density (that, with a slight abuse of notation, we still denote by~$m_0$) bounded and with compact support.
\end{itemize}
\begin{example} A coupling $F$ satisfying the above assumptions is of the form
$$F(x,t,m)=V\left(x,t,(\rho\star m)(x,t)\right),$$ where $\rho$ and $V$ are analytic functions with $\|V\|_{C^2}$ bounded. This assumption is very restrictive and it plays a role only in the proof of the uniqueness of optimal trajectories; we refer the reader to a forthcoming paper \cite{MMMT2} for much weaker hypotheses.
\end{example}

\vskip 5mm

We now introduce our definition of solution of the MFG system~\eqref{eq:MFG1} and state the main result concerning its existence.
\begin{definition}\label{defsolmfg}
The pair $(u,m)$ is a solution of system \eqref{eq:MFG1} if:
\begin{itemize}
\item[1)] $(u,m)\in W^{1,\infty}(\re^d\times[0,T])\times (C^0([0,T], \mathcal P_1)\cap L^{\infty}(\re^d\times[0,T]))$;
\item[2)] Equation~\eqref{eq:MFG1}-(i) is satisfied by $u$ in the viscosity sense;
\item[3)] Equation~\eqref{eq:MFG1}-(ii) is satisfied by $m$ in the sense of distributions.
\end{itemize}
\end{definition}
Here below we state the main result of this paper.
\begin{theorem}\label{thm:main}
 Under the above assumptions,  system \eqref{eq:MFG1} has a solution $(u,m)$ as in Definition~\ref{defsolmfg}. Moreover
 $m$ is the push-forward of $m_0$ through the characteristic flow
\begin{equation}\label{chflow}
x'(s)= -D_xu(x(s),s) \,B(x(s))\,B^T(x(s)),\quad x(0)=x.
\end{equation}
  \end{theorem}
For the sake of simplicity, in the following sections we prove Theorem \ref{thm:main} in the particular case $d=2$, $h_{11}=1$ and $h_{21}(x_1)\equiv 0$. Denoting $h(x_1):=h_{22}(x_1)$ the matrix $B$ is
\begin{equation}\label{GRU}
B(x)=
   \begin{pmatrix}
     1& 0\\
      0& h(x_1)        \end{pmatrix}
    \end{equation}
and the dynamics \eqref{eq:HJgen}
becomes:
\begin{equation}\label{eq:HJ2}
\left\{
\begin{array}{ll}
 x_1'(s)=\alpha_1(s)& \quad x_1(t)=x_1 \\
 x_2'(s)=h(x_1(s))\alpha_2(s)&\quad x_2(t)=x_2.
\end{array}\right.
\end{equation}
In this case the  Mean Field Game \eqref{eq:MFG1} is
\begin{equation}\label{eq:MFG2}
\left\{\begin{array}{lll}
(i)&\quad-\partial_t u+\frac 12 |D_Bu|^2=F(x,t, m)&\qquad \textrm{in }\re^2\times (0,T)\\
(ii)&\quad\partial_t m-\diver_B (m D_B u)=0&\qquad \textrm{in }\re^2\times (0,T)\\
(iii)&\quad m(x,0)=m_0(x), u(x,T)=G(x, m(T))&\qquad \textrm{on }\re^2,
\end{array}\right.
\end{equation}
where, for $x=(x_1, x_2)\in \re^2$, $\phi:\R^2\to\R$ and $\Phi:\R^2\to\R^2$  differentiable, we set
$$D_B \phi(x):=(\partial_{x_1}\phi(x) , h(x_1)\partial_{x_2}\phi(x)),\quad \diver_B\Phi(x) :=\partial_{x_1}\Phi_1(x)+h(x_1)\partial_{x_2}\Phi_2(x).$$
In this case we easily see that the direction along $x_2$ is forbidden when $h(x_1)$ has zero value (see also the results of \cite{MMMT2} concerning the same model under weaker assumptions).
\begin{example}\label{ex:Gru}
Examples of metric defined by \eqref{GRU}, are the Grushin type problems, with  analytic and bounded $h$, as $h(x_1)=\sin (x_1)$ or $h(x_1)=\frac{x_1}{\sqrt{1+x_1^2}}$ (see \cite{LM}).
\end{example}
\begin{example}\label{ex:heisen}
For $d=3$, the matrix $B(x)$ can be of Heisenberg type (see \cite{CiSa})
\begin{equation*}
B(x)=
   \begin{pmatrix}
     1& 0&0\\
      0& 1& 0\\
      h_{31}(x_1,x_2)&  h_{32}(x_1,x_2)&h_{33}(x_1,x_2)      \end{pmatrix}
    \end{equation*}
and the corresponding dynamics are
\begin{equation*}
\left\{
\begin{array}{ll}
x_1'(s)=\alpha_1(s)&\\
x_2'(s)=\alpha_2(s)&\\
x_3'(s)=h_{31}(x_1(s),x_2(s))\alpha_1(s)+h_{32}(x_1(s),x_2(s))\alpha_2(s)+h_{33}(x_1(s),x_2(s))\alpha_3(s).&
\end{array}\right.
\end{equation*}

\end{example}

This paper is organized as follows: in section~\ref{sect:HJ} we will find some properties of the solution of the Hamilton-Jacobi equation~\eqref{eq:MFG2}-(i) with fixed $m$. In section~\ref{sect:c_eq} we study the continuity equation~\eqref{eq:MFG2}-(ii) where~$u$ is the solution of a Hamilton-Jacobi equation found in the previous section.
In section~\ref{sect:MFG} we give the proof of the existence of the solution to system~\eqref{eq:MFG2}.
Finally, in the Appendix, we state the problem and the results for the general d-dimensional case ~\eqref{eq:MFG1} and we
introduce the notion of $B$-differentiability with the main properties of $B$-differentiable functions.

%
%
\section{The Hamilton-Jacobi equation}\label{sect:HJ}

 The aim of this section is to study the Hamilton-Jacobi equation~\eqref{eq:MFG2}-(i) with $m$ fixed, namely
\begin{equation}\label{eq:HJ1}
\left\{\begin{array}{ll}
-\partial_t u+\frac 1 2 ( (\partial_{x_1}u)^2+h(x_1)^2(\partial_{x_2}u)^2)=f(x, t)&\qquad \textrm{in }\re^2\times (0,T),\\
u(x,T)=g(x)&\qquad \textrm{on }\re^2
\end{array}\right.
\end{equation}
where~$f(x,t):=F(x,t, m)$ and $g(x):=G(x, m(T))$; hence, our assumptions (H1)-(H3) read
\begin{equation}\label{ass:HJ}
\textrm{$h$ and $f$ are analytic with $\|h\|_{C^2}+ \|f\|_{C^2}+\|g\|_{C^2}<C$.}
\end{equation}
In particular, we shall prove several regularity properties of the solution (Lipschitz continuity and semiconcavity) and mainly the uniqueness of optimal trajectories. In our opinion this uniqueness result has its own interest because, as far as we know, this property has never been tackled before for non-coercive dynamics as~\eqref{eq:HJgen}.\\
The solution $u$ of \eqref{eq:HJ1} can be represented as the value function of the following control problem. Let the set of the admissible controls~$\cal{A}$ be defined as in~\eqref{A} with $d=2$ and, for each control $\alpha\in\mathcal{A}$, let $x(\cdot)$ be the trajectory given by~\eqref{eq:HJ2}; we define the cost
\begin{equation}\label{J}
J_t(\alpha(\cdot),x(\cdot)):= \int_t^T\frac12 |\alpha(\tau)|^2+f(x(\tau),\tau)\,d\tau+g(x(T)).
\end{equation}
\begin{definition} The value function for the cost $J_t$ in \eqref{J} is
\begin{equation}\label{repr}u(x,t):=\inf\left\{ J_t(\alpha(\cdot),x(\cdot)):\, \alpha(\cdot)\in  {\cal{A}},\, ( x (\cdot),\alpha(\cdot)) \text{ satisfying }\eqref{eq:HJ2}\right\}.
\end{equation}
\end{definition}
\begin{lemma}Under assumptions~\eqref{ass:HJ}, the value function $u(x,t)$
is the unique bounded uniformly continuous viscosity solution to problem~\eqref{eq:HJ1}.
\end{lemma}
\begin{proof}
We argue as in Lemma 2.1 of \cite{MMMT2}.
\end{proof}

\begin{lemma} \label{L1}
Under assumptions~\eqref{ass:HJ}, there hold:
\begin{enumerate}
\item The function $u$ defined in \eqref{repr} is bounded in $\re^2\times [0,T]$,
\item $u(x,t)$ is Lipschitz continuous with respect to the spatial variable~$x$,
\item $u(x,t)$ is Lipschitz continuous with respect to the time variable~$t$.
\end{enumerate}
\end{lemma}
\begin{proof}  
We argue as in Lemma 2.2 of \cite{MMMT2}.
\end{proof}

\begin{lemma} Under assumptions~\eqref{ass:HJ}, the value function~$u(x,t)$, defined in~\eqref{repr} is semiconcave with respect to the variable~$x$.
\end{lemma}
\begin{proof}
We argue as in Lemma 2.3 of \cite{MMMT2}.
\end{proof}
For any $(x,t)\in\re^2\times [0,T]$, we denote by ${\cal A}(x,t)$  the \textbf{set of optimal controls} of the minimization problem~\eqref{repr} whose trajectories are governed by~\eqref{eq:HJ2}. As in \cite{C}, we easily see that
if $(t_n, x_n)\to (x,t)$ and $\alpha_n\in{\cal A}(x_n, t_n)$, then, possibly passing to some subsequence, $\alpha_n$ weakly converges in $L^2$ to some $\alpha\in{\cal A}(x,t)$.

The following result gives the optimality condition:
\begin{lemma}\label{EL}
Let $\alpha:=(\alpha_1,\alpha_2)\in {\cal A}(x,t)$.  Let  $x(\cdot):=(x_1(\cdot), x_2(\cdot))$ be the associated optimal trajectory governed by \eqref{eq:HJ2} under assumptions \eqref{ass:HJ}; then the following properties hold:
\begin{enumerate}
\item The optimal control $\alpha$ satisfies
\begin{equation}\forall s\in [t,T]\quad
\label{sistemalphap}
\left\{\begin{array}{ll}
 \alpha_1(s)= p_1(s) \\\alpha_2(s)= h(x_1(s))p_2(s)
\end{array}\right.
\end{equation}
where $p:=(p_1,p_2):[t,T]\to\R^2$  satisfies the following implicit equations:
\begin{equation}\forall s\in [t,T]\quad
\!\begin{cases}\label{tag:p}
p_1(s)\!\!&\!\!\!:=-g_{x_1}(x(T))-\!\!\displaystyle\int_s^T \!\!f_{x_1}(x,\tau)-p_2^2h'(x_1)h(x_1)\,d\tau,\\
p_2(s)\!\!&\!\!\!:=-g_{x_2}(x(T))-\displaystyle\int_s^Tf_{x_2}(x,\tau)\,d\tau.\\
\end{cases}
\end{equation}
\item The control $\alpha$ is of class $C^1$, as well as the pair
$(x, p)=((x_1,x_2), (p_1,p_2))$, and the latter  satisfies the system of differential equations:
\begin{equation}\label{system}
\begin{cases}
(1)\qquad x_1'&= p_1 \\
(2)\qquad x_2'&= h^2(x_1)p_2 \\
(3)\qquad p_1'&= -p_2^2h'(x_1)h(x_1)+f_{x_1}(x,s) \\
(4)\qquad p_2'&= f_{x_2}(x,s),
\end{cases}
\end{equation}
with the mixed boundary conditions $x(t)= x$, $p(T)=-\nabla g(x(T))$.
\end{enumerate}
\end{lemma}
\begin{proof}
We argue as in Proposition 2.1 of \cite{MMMT2}.
\end{proof}
By standard arguments, one can prove the following Lemma so we omit the proof.
\begin{lemma}
\label{concatenato}
Let $\alpha_\star$ be optimal control in~${\cal A}(x,t)$ and $x_\star(\cdot)$ be  the corresponding optimal trajectory for $J_t$. Let $\widetilde\alpha(\cdot) \in {\cal A}(x_\star(s),s)$.
The control law
\begin{equation*}
\overline \alpha(\tau):=\left\{
\begin{array}{ll}
\alpha_\star(\tau)& \; \hbox{ if }\tau \in[t,s]\\
      {\widetilde{\alpha}}(\tau) &\;  \hbox{ if } \tau \in[s, T]
\end{array}\right.
\end{equation*}
is optimal for  $u(x,t)$, i.e., $\overline \alpha\in {\cal A}(x,t)$.
\end{lemma}

In the following Lemma~\ref{CK} we show that the  optimal trajectories are analytic; this property will play a crucial role in the proof of uniqueness of optimal trajectories (after the starting time) which is established in the next proposition.
\begin{lemma}\label{CK}
Under assumptions \eqref{ass:HJ}, the solutions $(x,p)$ of the system \eqref{system} are analytic on $]t,T[$.
\end{lemma}
\begin{proof}  The system \eqref{system} is of the form
$$ (x_1,x_2,p_1, p_2)'=F(s,x_1,x_2,p_1,p_2),$$
where $F:[t,T]\times\R^4\to\R^4$ is the analytic function
$$F(s, x_1,x_2,p_1,p_2):=(p_1, h^2(x_1)p_2, -p_2^2h'(x_1)h(x_1)+f_{x_1}(x_1,x_2,s), f_{x_2}(x_1,x_2,s)).$$
The Cauchy-Kovalevskaya Theorem \cite[Theorem 2.2.21]{RR} and the Cauchy-Lipschitz theorem yield the conclusion.
\end{proof}

\begin{proposition}\label{thmunicita}
Under assumptions \eqref{ass:HJ}, for any $\alpha_\star\in {\cal{A}}(x,t)$, let~$x_\star(\cdot)$ be the corresponding optimal trajectory.
\begin{enumerate}
\item For every $s\in(t,T]$, there are no optimal trajectories  for $J_s(\alpha(\cdot),x(\cdot))$ starting from $x_\star(s)$ at time $s$ other than
 $x_\star(\cdot)$, restricted to $(s,T]$.
\item For every $s\in (t,T)$, $D_Bu(x(s),s)$ exists if and only if ${ \cal{A}}(x(s),s)=\{\alpha\}$ is a singleton and $D_Bu(x(s),s)=-\alpha(s)$ (i.e., $u_{x_1}(x(s),s)=-\alpha_1(s)$, $h(x_1(s))u_{x_2}(x(s),s)=-\alpha_2(s)$).
\end{enumerate}
\end{proposition}

\begin{proof}
1. Let us suppose that there exists another optimal trajectory $\widetilde x(\cdot)$ for $[s,T]$.
From Lemma \ref{concatenato} the trajectory $\overline x$ obtained linking $x(\cdot)$ in $[t,s]$ with
 $\widetilde x(\cdot)$ on $[s,T]$ is optimal whence, from
 Lemma \ref{CK}, analytic on $]t,T[$ as well as $x_{\star}(\cdot)$. Since $x_{\star}(\cdot)$ and $\overline x(\cdot)$ coincide on the non trivial interval $]t,s]$, it follows from the Identity Theorem for analytic  functions that $x_{\star}(\cdot)=\overline x(\cdot)$ on $[t,T]$.  Therefore,  their restrictions $x_{\star}(\cdot)$ and $\widetilde x(\cdot)$ to $[s,T]$ are equal.\\
2. We observe that a concatenation result holds (see Appendix A of \cite{MMMT2}); hence by
Point 1, arguing as in \cite{C},  
we get that~$u(\cdot, s)$ is always $B$-differentiable (see the Definition~\ref{G-differenz} below) at $x(s)$ for any $s\in (t,T)$ and $D_Bu(x,s)=-\alpha(s)$.
Then we complete the proof as in Lemma 4.9 and Remark 4.10 of \cite{C}.
\end{proof}

\begin{lemma}
Let $x(\cdot):=(x_1(\cdot),x_2(\cdot))$ be an absolutely continuous solution of the problem
\begin{equation}
\label{4.11}
\left\{
\begin{array}{ll}
x_1'(s)=-u_{x_1}(x(s),s),& \quad x_1(t)=x_1, \\
x_2'(s)=-h^2(x_1(s))u_{x_2}(x(s),s),&\quad x_2(t)=x_2,
\end{array}\right.
\end{equation}
then the control $\alpha=(\alpha_1,\alpha_2)$, with
$$\alpha_1(s)=-u_{x_1}(x(s),s), \ \alpha_2(s)=-h(x_1(s))u_{x_2}(x(s),s)$$ is optimal for $u(x,t)$.
In particular if $u(\cdot, t)$ is differentiable at $x$ then problem \eqref{4.11} has a unique solution corresponding to the optimal trajectory.
\end{lemma}
\begin{proof}
We argue as in Lemma 3.6 of \cite{MMMT2}.
\end{proof}

\section{The continuity equation}\label{sect:c_eq}
In this section we want to study the well posedness of the problem
\begin{equation}
\label{continuity}
\left\{
\begin{array}{ll}
\partial_t m-\diver_B (m\, D_B u)=0&\qquad \textrm{in }\re^2\times (0,T)\\
m(x,0)=m_0(x) &\qquad \textrm{on }\re^2,
\end{array}\right.
\end{equation}
where $u$ is a solution to the Hamilton-Jacobi problem
\begin{equation}\label{HJ}
\left\{\begin{array}{ll}
-\partial_t u+\frac12 |D_Gu|^2=F(x, t, \overline {m})&\qquad \textrm{in }\re^2\times (0,T),\\
u(x,T)=G(x, \overline {m}(T)),&\qquad \textrm{on }\re^2,
\end{array}\right.
\end{equation}
where the function~$\overline m$ is fixed and fulfills
\begin{equation}\label{mcnd}
\overline {m}\in C^{1/2}([0,T],\mathcal P_1),\qquad \int_{\re^2}|x|^2\,d\overline m(t)(x)\leq K, \qquad t\in[0,T],
\end{equation}
and $m_0$ fulfills assumption~(H4).
The aim of this Section is to prove that problem~\eqref{continuity} has a unique solution which, moreover, can be characterized as the push-forward of the measure $m_0$ through the flow defined by the optimal trajectories of the control problem underlying \eqref{eq:HJ1}.
\begin{proposition}\label{prp:m}
Under assumption~\eqref{ass:HJ} and~$(H4)$,
problem \eqref{continuity} has a unique bounded solution $m$ in the sense of Definition \ref{defsolmfg}. Moreover $m(t,\cdot)$ is a measure absolutely continuous with respect to the Lebesgue measure such that
$\sup_{t\in[0,T]}\|m(t,\cdot)\|_{\infty}\leq C$ and it is a Lipschitz continuous map from $[0,T]$ to $\mathcal P_1$ with a Lipschitz constant bounded by $\|Du\|_\infty \|h^2\|_\infty$.
Moreover, the function $m$ satisfies:
\begin{equation}
\label{ambrosio}
\int_{\re^2} \phi\, dm(t)=\int_{\re^2}\phi(\overline {\gamma}_x(t))\,m_0(x)\, dx \qquad \forall \phi\in C^0_0(\R^2), \, \forall t\in[0,T]
\end{equation}
where, for a.e. $x\in\re^2$,  $\overline{\gamma}_x$ is the solution to \eqref{4.11}.
\end{proposition}
\begin{proof}
The steps to achieve the statement are similar to that of Section 3 of \cite{MMMT2}.
\end{proof}

%
%

\section{Proof of the main Theorem}\label{sect:MFG}
This section is devoted to the proof of our main Theorem~\ref{thm:main}. To this end, it is expedient to recall the stability result (as \cite[Lemma 4.19]{C}) which still holds in our case:
\begin{lemma}\label{lemma4.19}
Let $m_n\in C([0,T], {{\cal P}_1})$ be uniformly convergent to $m\in C([0,T], {{\cal P}_1})$. Then the solution $u_n$ of~\eqref{eq:HJ1} with $f_n(x,t)=F(x,t,m_n)$ and $g(x)=G(x,m_n(T))$ converges locally uniformly to the solution~$u$ to~\eqref{eq:HJ1}.\\
Moreover, denote $\mu_n(s):=\Phi_n(\cdot,0,s)\sharp m_0$ and~$\mu(s):=\Phi(\cdot,0,s)\sharp m_0$ where~$\Phi_n$ (respectively,~$\Phi$) is the flow associated to~$u_n$ (resp.,~$u$). Then, $\mu_n\to\mu$ in $C([0,T], {{\cal P}_1})$.
\end{lemma}

\begin{proof}[Theorem~\ref{thm:main}]
The proof follows the lines of the proof of Theorem 1.1 in \cite{MMMT2}.
\end{proof}

\section*{Appendix}\label{sect:Gdiff}
\subsection{The general case}\label{sect:gen}
In this subsection we collect the ingredients to prove Theorem \ref{thm:main} in the general case as stated in the Introduction. We find the system satisfied by $(x,p)$ analogous to \eqref{system}, we deduce the analyticity of the solution and then the uniqueness of the optimal trajectory after the starting time as in Proposition \ref{thmunicita}. The proofs rely on the same arguments of the model problem studied in the sections above hence we only emphasize the main differences.\\
Following the same procedure as in Section \ref{sect:HJ}, we can prove that the solution $u$ of the Hamilton-Jacobi equation \eqref{eq:MFG1}-(i) with $m$ fixed is bounded in $\re^d\times[0,T]$, Lipschitz continuous with respect to $x$ and $t$ and it is semiconcave with respect to $x$.
Moreover if $\alpha$ and $x$ are respectively the optimal control and the trajectory of the control problem \eqref{Jgen}-\eqref{eq:HJgen} then, as in Lemma \ref{EL}, we can prove:
\begin{enumerate}
\item  The pair $(x, p)$ is of class $C^1$ and satisfies the system of differential equations:
\begin{equation}\label{systemgen}
\begin{cases}
(1)\qquad  x'= p\,B(x)\,B^T(x), \\
(2)\qquad p'= - \dfrac{D_x|pB(x)|^2}2+  D_xf(x,s)\\
(3)\qquad x(t)= x,\ p(T)=-\nabla g(x(T)).
\end{cases}
\end{equation}
Since the functions $h_{ij}$ in the matrix $B$ do not depend on the last variable $x_d$, then, in our case,  the last coordinate of   $D_x|p(s)B(x)|^2$ is $0$, i.e.
in (2) of \eqref{systemgen} the last equation for $p$ is $p_d'(s)= f_{x_d}(x,s)$.
\item The optimal control $\alpha$ is of class $C^1$ and is given by
\begin{equation*}\forall s\in [t,T]\quad
\quad \alpha(s)= p(s)B(x(s)).
\end{equation*}
\end{enumerate}
As in Lemma \ref{CK} we can prove that, under our assumptions, the solutions of  \eqref{systemgen} are analytic and we can use the same argument to prove the statement of Proposition \ref{thmunicita} where the equalities in Point 2 are replaced by
$$\forall s\in (t,T)\quad \alpha(s)=-D_xu(x(s),s)B(x(s)),$$
and from \eqref{eq:HJgen}, equation \eqref{systemgen}-(1)
becomes
$$x'(s)= -D_xu(x(s),s) \,B(x(s))\,B^T(x(s))\,.$$
By the uniqueness of the optimal trajectory after the starting time, we can fix a Borel measurable selection $\overline\alpha$ of ${\cal A}(x,t)$; we define the flow $\Phi(x,t,s)$
$$\Phi(x,t,s)= x+ \int_t^s\,\alpha(\tau)B^T(x(\tau))\,d\tau,$$
where $\alpha$ is the selected element in ${\cal A}(x,t)$,
which satisfies
$$\partial_s\Phi(x,t,s)=-D_xu(x(s),s)\,B(x(s))\, B^T(x(s)).$$

The main issue for extending the results of Section \ref{sect:c_eq} is the definition of the approximating smooth sequence $b_{\epsilon}$.
In this case taking
$$
b^{\varepsilon}(x,t):=-\frac{1}{\mu^{\varepsilon}}(\mu\,\partial_pH(x,D_xu))\star  \rho_{\varepsilon}$$
we can obtain the same results as in Proposition \ref{prp:m}.\\
Collecting all the results recalled here we prove also in this general case Theorem \ref{thm:main}.

%
%
\subsection{B-differentiability}\label{sect:Bdiff}
In this subsection we give the general definition of $B$-diffe\-rentia\-bility in $\re^d$ which is used along the paper.
For the main properties for semiconcave functions we refer to the Appendix B in \cite{MMMT2} where the 2-dimensional case is considered.

\begin{definition}
\label{G-differenz}
A function~$u:\re^d\rightarrow \re$ is $B$-differentiable in $x\in\re^d$ if there exists $\rho_B\in \re^d$ such that
\[
\lim_{v\rightarrow 0}\frac{u(\tilde x)-u(x)-(\rho_B,v)}{\vert v\vert}=0;
\]
where, for $v\in\re^d$ we iteratively define
$\tilde x_1=x_1+h_{11} v_1$, and $\tilde x_i=x_i+\sum_{j=1}îh_{ij}(\tilde x_1, \cdots, \tilde x_{i-1})v_j$,
where $h_{ij}$ are defined in \eqref{hmatrix}.
\end{definition}
In the case treated in the previous sections definition \ref{G-differenz} becomes
\[
\lim_{v\rightarrow 0}\frac{u(x_1+v_1, x_2+h(x_1+v_1)v_2)-u(x_1, x_2)-(\rho_B,v)}{\vert v\vert}= 0.
\]

\begin{remark}
If $u$ is differentiable then $\rho_B=:D_Bu=Du\,B$, where $B$ is defined in \eqref{hmatrix}.
\label{D+D-}
\end{remark}

\noindent{\bf Acknowledgments.} The first and second authors are members of GNAMPA-INdAM and were partially supported also by the research project of the University of Padova "Mean-Field Games and Nonlinear PDEs" and by the Fondazione CaRiPaRo Project "Nonlinear Partial Differential Equations: Asymptotic Problems and Mean-Field Games". The fourth author has been partially funded by the ANR project ANR-16-CE40-0015-01.

\end{document}